\documentclass[11pt]{article}
\usepackage{amsmath,amssymb,amsthm}
\usepackage{fullpage}
\usepackage{graphicx}
\newtheorem{prop}{Proposition}[section]
\newtheorem{theorem}[prop]{Theorem}

\newtheorem{corollary}[prop]{Corollary}

\newtheorem{lemma}[prop]{Lemma}
\newtheorem{proposition}[prop]{Proposition}

\newtheorem*{defn*}{Definition}
\theoremstyle{definition}
\newtheorem{remark}{Remark}

\title{On the Imaginary Part of \(\zeta'/\zeta\)}

\author{Farzad Aryan}
\date{}

\begin{document}
\maketitle
\begin{abstract}

We derive an explicit zero-side formula for the imaginary part of a smooth truncation of the logarithmic derivative of the Riemann zeta-function. The resulting kernel is the odd companion of the Fejér kernel in Montgomery's pair-correlation identity. Motivated by the recent breakthrough of Anthropic's Claude \cite{Claude26}, we also give an expository account of how the real-part identity can be turned into a matrix problem, overcoming the clustering obstruction for zeros off the critical line. Finally, we discuss how the new imaginary-part identity may contain additional information about such clustering.
\end{abstract}
\section{Introduction}

Montgomery's famous pair-correlation~\cite{Montgomery} identity may be viewed as an average of
the real part of a smoothly truncated logarithmic derivative of the Riemann
zeta-function. More precisely, let
\[
Z(s):=-\frac{2}{\log T}\sum_{n<T}\frac{\Lambda(n)}{n^s}
\left(1-\frac{\log n}{\log T}\right),
\qquad
\mathcal C(t):=\Re Z\!\left(\frac12+it\right).
\]
Under the Riemann Hypothesis, \(\mathcal C(t)\) admits the zero-side
representation
\begin{equation}
\label{C_t_zero}
\mathcal C(t)
=
-1+
\sum_{\gamma'}
\left(
\frac{\sin\!\big(\tfrac12(\gamma'-t)\log T\big)}
{\tfrac12(\gamma'-t)\log T}
\right)^2
+
O\!\left(\frac{1}{\log T}\right).
\end{equation}

This gives a useful interpretation of \(\mathcal C(t)\). Up to the constant
\(-1\), it is a smoothed count of the zeros lying within approximately
\(1/\log T\) of \(t\). In particular, if \(t=\gamma\) is the ordinate of a
multiple zero, every copy of that zero contributes \(1\) to the sum. Thus a
zero of multiplicity \(m\) contributes \(m-1\) to \(\mathcal C(\gamma)\)
before the contribution of nearby zeros is even taken into account. This
makes \(\mathcal C\) a natural quantity for studying both multiplicities and
the local spacing of zeros.

Averaging over the ordinates of the zeros gives Montgomery's pair-correlation
identity,
\begin{equation}
\label{C_t}
\frac{1}{N(T)}\sum_{0<\gamma'\leq T}\mathcal C(\gamma')
\sim \frac13.
\end{equation}
One of the main consequences of \eqref{C_t} is, assuming RH, at least
two-thirds of the zeros are simple. 

In \cite{Ary22}, the author proved an unconditional analogue of this smoothed
identity, allowing the zeros to move away from the critical line. When zeros
are allowed to move off the critical line, however, the zero-side picture
changes substantially. The quantities corresponding to the Fej\'er-kernel
terms, in \eqref{C_t_zero}, are now evaluated at complex arguments. Individual contributions can
therefore be complex and can become very large when a zero moves away from the
critical line. The positivity used in Montgomery's argument is lost, and
clustering among nearby zeros can produce substantial cancellation.

It was therefore natural to ask whether Montgomery's \(2/3\) result could
nevertheless be recovered unconditionally. This clustering obstruction was
recently overcome in \cite{Claude26}. Building on the unconditional
pair-correlation framework of \cite{Ary22} and its subsequent developments
\cite{BGSTB24,BGSTB25,GS25,GS26}, Claude replaced the termwise positivity
argument by a global linear-algebraic analysis.\footnote{A shorter proof was subsequently given by Lamzouri \cite{Lamzouri26}, replacing the finite-dimensional matrix framework by a Hilbert-space inequality.} This yields an unconditional
result that more than two-thirds of the zeros are simple and lie on the
critical line. We explain the argument and the role of clustering in section~\ref{Expo}.

\subsection{The imaginary-part analogue}

The present paper is motivated by a related question. The prime-side definition of
\(Z(s)\) gives its real and imaginary parts immediately. The difficulty is to
find a corresponding formula on the zero side for
\[
\Im Z\!\left(\frac12+it\right).
\]

This question first arose during the work leading to \cite{Ary22}. A further motivation came from the study of higher moments of the distribution of zeta zeros \cite{Ary20HigherMoments}, where the imaginary part appears naturally. However, the argument that produces the zero-side formula for \(\mathcal C(t)\) does not appear to adapt directly to the imaginary part, and at the time no satisfactory analogue was found.

More recently, the author found that a combinatorial identity gives a
different route to the problem. This leads to an explicit zero-side formula
for the imaginary part of the same smoothly truncated logarithmic derivative.

The recent breakthrough \cite{Claude26} makes this identity particularly
timely. It shows that the real-part identity contains considerably more
information about the zeros than could be extracted by the classical
positivity argument alone. It therefore seems important to develop the corresponding imaginary-part identity and make it available for the study of multiplicities, spacing, and higher moments of the zeros.\footnote{While this paper was being completed, OpenAI released an AI-generated proof of the quasi-Riemann hypothesis, establishing the zero-free half-plane \(\operatorname{Re}(s)>7/8\); the proof is accompanied by a formalization in Lean \cite{OpenAIQRH26}.}

\section{An Identity for the Imaginary Part}
We now state the identity for the imaginary part as a theorem.
\begin{theorem}[Imaginary-part kernel identity]
\label{thm:ImZ-limit-kernel}
Let \(T>e\), put \(L=\log T\), and let \(t\asymp T\).  Define
\[
A_\infty(z)
:=
2\int_0^1(1-u)\sin(uz)\,du
=
\frac{2(z-\sin z)}{z^2},
\qquad A_\infty(0)=0.
\]
Then
\[
\boxed{
\Im Z\!\left(\frac12+it\right)
=
\Re\sum_{\rho}
A_\infty\!\left(
-iL\left(\rho-\frac12-it\right)
\right)
+
O\!\left(\frac1L\right),
}
\]
where the sum is over the non-trivial zeros of \(\zeta\), counted with
multiplicity and taken symmetrically.

In particular, under the Riemann Hypothesis this reduces to
\[
\Im Z\!\left(\frac12+it\right)
=
\sum_{\gamma}
A_\infty\!\bigl((\gamma-t)L\bigr)
+
O\!\left(\frac1L\right).
\]
\end{theorem}

\begin{remark}
The kernel \(A_\infty\) is the odd companion of the Fej\'er kernel appearing
in the real-part identity. Indeed,
\[
\left(\frac{\sin(z/2)}{z/2}\right)^2
=
2\int_0^1(1-u)\cos(uz)\,du,
\]
whereas
\[
A_\infty(z)
=
2\int_0^1(1-u)\sin(uz)\,du.
\]
Thus the real- and imaginary-part kernels arise respectively as the cosine
and sine transforms of the same triangular weight. In particular,
\(A_\infty\) is odd, which explains why its unshifted average over pairs of
zeros exhibits cancellation.
\end{remark}

\begin{corollary}[Shifted average at the zeros]
\label{cor:shifted-imaginary}
Let \(\varepsilon>0\) be fixed. Then, for fixed \(d\in\mathbb R\),
\[
\frac{1}{N(T)}
\sum_{0<\gamma<T}
\Im_{1-\varepsilon}\!\left(
\gamma-\frac{2\pi d}{\log T}
\right)
\sim
\int_0^1
2u(1-u)\sin(2\pi du)\,du.
\]
Equivalently, writing \(\alpha=2\pi d\),
\[
\frac{1}{N(T)}
\sum_{0<\gamma<T}
\Im_{1-\varepsilon}\!\left(
\gamma-\frac{\alpha}{\log T}
\right)
\sim
\frac{2\bigl(2-2\cos\alpha-\alpha\sin\alpha\bigr)}
{\alpha^3}.
\]
\end{corollary}
\begin{remark}
For \(d=0\), the average in the corollary vanishes because \(A_\infty\) is odd.
A microscopic shift therefore reveals information that is invisible when one
evaluates the imaginary part exactly at the zeros. As already observed in
\cite{Ary20HigherMoments}, the main term reaches its maximum at approximately
\[
d=0.4147.
\]

This is interesting in comparison with the real-part identity. The
corresponding average of \(\mathcal C\) is largest at \(d=0\), whereas for the
imaginary part the maximum occurs at a non-zero microscopic shift. Thus the
real and imaginary parts appear to probe the local zero configuration in
genuinely different ways.
\end{remark}
Before turning to the proof of the imaginary-part identity, we pause to explain the recent development that motivates much of the present discussion. Since the significance of the imaginary-part analogue is best understood in light of what has recently been extracted from the real-part identity, we give a self-contained account of Claude's argument and of how it overcomes the clustering obstruction.

\section{An Expository Note on Claude's Argument}
\label{Expo}
Around 2017, the author \cite{Ary22} observed that an analogue of Montgomery's identity \eqref{C_t} could be obtained without assuming the Riemann Hypothesis, initially with a Gaussian weight.\footnote{Gaussian weights appear to be better suited to questions concerning zero-free regions. Very recently, the known zero-free region was significantly improved by OpenAI, which established the zero-free half-plane \(\operatorname{Re}(s)>7/8\) \cite{OpenAIQRH26}.}
 It was later realized that the Fej\'er kernel was more natural for the present problem, since it captures the microscopic behavior of the zeros more directly. 

As mentioned in the introduction, an important application of Montgomery's theorem is that, assuming RH, at least two-thirds of the zeros of the Riemann zeta-function are simple. Let us briefly recall the argument.\\

Let \(m_\gamma\) denote the multiplicity of the zero with ordinate \(\gamma\). Then
\begin{equation}
\label{Nt}
    \sum_{\gamma\leq T} m_\gamma
=
N(T)
\sim
\frac{T}{2\pi}\log T.
\end{equation}

So the first moment of the sequence $\{m_\gamma\}$ is known. A natural way to measure how far the zeros are from being simple is then through the second moment
\[
\sum_{\gamma\leq T}m_\gamma^2.
\]

\noindent Under RH, the identity \eqref{C_t}, together with the zero-side formula for \(\mathcal C(\gamma)\), can be written schematically as
\[
\sum_{\gamma\leq T}m_\gamma^2
+
\text{positive pair-correlation contribution}
=
\left(\frac43+o(1)\right)N(T).
\]

\noindent Using the positivity of the pair-correlation contribution, we obtain
\[
\sum_{\gamma\leq T}m_\gamma^2
\leq
\left(\frac43+o(1)\right)N(T).
\]

At this point we know the first moment of the sequence \(\{m_\gamma\}\) and have
an upper bound for its second moment. Subtracting the first moment from the
second moment, we get
\[
\sum_{\gamma\leq T}m_\gamma(m_\gamma-1)
=
\sum_{\gamma\leq T}m_\gamma^2-N(T)
\leq
\left(\frac13+o(1)\right)N(T).
\]

For a simple zero, the contribution to the sum on the left is \(0\). For a
multiple zero, \(m_\gamma\geq2\), and
\[
m_\gamma(m_\gamma-1)\geq m_\gamma.
\]

Therefore the number of multiple zeros, counted with multiplicity, is at most
\[
\left(\frac13+o(1)\right)N(T).
\]

Hence at least
\[
\left(\frac23-o(1)\right)N(T)
\]
of the zeros are simple. \\

After the unconditional pair-correlation identity was established in \cite{Ary22}, and subsequently developed in \cite{BGSTB24,BGSTB25,GS25,GS26}, it became natural to ask whether Montgomery's \(2/3\) simple-zero argument itself could be made unconditional. Parts of Montgomery's argument survive without RH, but a major obstruction arises from the clustering of zeros. The recent argument discovered by Claude \cite{Claude26}, and subsequently verified and communicated by Alp\"oge and Furman, overcame this obstruction and recovered the \(2/3\) conclusion unconditionally.

\subsection{The clustering barrier}

In this subsection we explain why the unconditional version of Montgomery's
identity did not immediately lead to an unconditional recovery of the
\(2/3\) simple-zero result. The main difficulty is the loss of positivity.
Under RH, all arguments of the Fej\'er kernel are real, so every term
$\operatorname{sinc}^2(x)$
is non-negative. Without RH, these arguments are generally complex. Although
certain symmetric combinations remain real, they need not remain positive.
This is the source of the clustering obstruction.\\

Suppressing the smoothing and taking \(\alpha=1\), the unconditional identity
of \cite{Ary22} takes the form
\begin{equation}
\label{Ary_Uncond}
    \sum_{\rho,\rho'}
\left(
\frac{
\sin\!\left(\frac{\rho-\rho'}{2i}\log T\right)
}{
\frac{\rho-\rho'}{2i}\log T
}
\right)^2
=
\left(\frac43+o(1)\right)N(T),
\end{equation}
where the sine is now understood as a complex-valued function.

Let $m_\rho$ denote the multiplicity of a zero $\rho$. Separating the terms
for which $\rho=\rho'$, we can write this as
\[
\sum_\rho m_\rho^2
+
\sum_{\rho\neq\rho'}
m_\rho m_{\rho'}
\left(
\frac{
\sin\!\left(\frac{\rho-\rho'}{2i}\log T\right)
}{
\frac{\rho-\rho'}{2i}\log T
}
\right)^2
=
\left(\frac43+o(1)\right)N(T).
\]

Thus the first term is exactly the second moment that appeared in
Montgomery's argument. What we have lost is the positivity of the second
term. Under RH every difference $(\rho-\rho')/i$ is real, and the Fej\'er
kernel is non-negative. Away from RH this is no longer true. \\

To see the clustering problem clearly, we first consider an isolated pair of zeros reflected across
the critical line,
\[
\rho_\pm
=
\frac12\pm\delta+i\gamma.
\]
Thus, for fixed \(\delta>0\), the interaction between these two zeros is
\[
\operatorname{sinc}^2\!\left(
\frac{\rho_+-\rho_-}{2i}\log T
\right)
\sim
\frac{T^{2\delta}}
{4\delta^2(\log T)^2}.
\]
This shows that an isolated off-line pair creates a contribution growing like $T^{2\delta}$, 
up to logarithmic factors. Thus an isolated reflected pair gives a large
positive contribution. This actually helps the argument, and was one reason
it initially looked plausible that zeros away from the critical line should be
very easy to detect.

\medskip

\noindent\textbf{Clustering.}
The problem begins when many zeros form a cluster. A slightly more careful
analysis shows that interactions among off-line zeros can still be handled
after grouping zeros with their reflections. The real obstruction comes from
their interaction with zeros on the critical line, where negative terms can
appear. We give a concrete example of this phenomenon in
Appendix~\ref{app:clustering-example}.

This is the clustering barrier. Once off-line zeros interact with sufficiently
many nearby critical-line zeros, the mixed negative terms can cancel the
positive information coming from the off-line pairs. The identity still
contains the second moment
\[
\sum_\rho m_\rho^2,
\]
but we can no longer simply discard the remaining terms and obtain an upper
bound for it. This is exactly where Montgomery's positivity argument breaks
down.

\subsection{How Claude bypasses the clustering barrier}

\textbf{Summary of the argument.} Claude's key idea is to replace the scalar contribution of each zero in \eqref{Ary_Uncond} by a
vector \(v_\rho\), and then use these vectors to construct a matrix \(G\).
Each vector \(v_\rho\) encodes both the height of the zero and, when RH
fails, its displacement from the critical line. The vectors are chosen so
that their inner products reproduce the sinc kernel. Consequently, the
squares of these inner products reproduce the kernel in Montgomery's
identity and, when zeros move off the critical line, the kernel appearing
in the unconditional identity \eqref{Ary_Uncond}. Forming
\[
G=\sum_\rho v_\rho v_\rho^T,
\]
the analytic number theory gives
\[
\operatorname{tr}G=(1+o(1))N(T),
\qquad
\operatorname{tr}(G^2)
=
\left(\frac43+o(1)\right)N(T).
\]

These two trace identities force \(G\) to have many positive eigenvalues, or
equivalently many positive directions. The zero-side structure then tells us
how efficiently different kinds of zeros can produce such directions. A
simple zero on the critical line can contribute one positive direction. A
multiple zero repeats the same direction, while a reflected pair of zeros off
the critical line contributes at most one positive direction for two zeros.
Thus multiple and off-line zeros waste positive directions. Since the trace
information forces \(G\) to have many positive directions, there cannot be
too many zeros of these types. This is the basic mechanism behind Claude's
argument.\\

Claude's argument succeeds because it does not only measure the size of the
zero contributions; it also keeps track of the directions in which those
contributions act. The real Fej\'er kernel, however, does not retain all of
the information about how these directions are arranged relative to one
another. The imaginary-part identity proved in this paper gives an additional,
sign-sensitive quantity which may help distinguish different arrangements of
the same directions. In particular, it may provide more information about how
the positive and negative directions coming from off-line zeros align and
cancel. If this information can be incorporated into the matrix construction,
it could lead to stronger restrictions on clustering than those obtained from
\(\operatorname{tr}G\) and \(\operatorname{tr}(G^2)\) alone. We do not pursue
this possibility here.\\

Here and below, by a positive direction of \(G\) we mean a nonzero real
vector \(x\) for which
\[
x^T Gx>0.
\]
Equivalently, the number of independent positive directions is the number of
positive eigenvalues of \(G\). We now describe precisely how the matrix \(G\) is
constructed and how different types of zeros contribute to these directions.

\paragraph{Composition of the matrix \(G\).} To understand the matrix, let us first understand the vector \(v_\rho\)
attached to a zero. Following \cite{Claude26}, write
\[
\gamma_\rho:=\frac{\rho-\frac12}{i}.
\]
Thus, if 
\[\rho=\frac12+\delta+i\gamma, \text{  then }
\gamma_\rho=\gamma-i\delta.\] 

The zeros considered in the construction of \(G\) have heights around \(T\), say in the interval \([T,2T]\). Claude divides this interval into a grid whose spacing is approximately the average gap between consecutive zeros of the Riemann zeta-function. Since the average gap near height \(T\) is about

$$
\frac{2\pi}{\log T},
$$

we may write the grid points as

$$
\alpha_k=T+\frac{2\pi k}{L},
\qquad L\asymp\log T.
$$

The vector \(v_\rho\) associated with a zero \(\rho\) can then be thought of as describing where that zero is located relative to this grid. Its \(k\)-th coordinate measures how strongly the zero is seen from the grid point \(\alpha_k\). 
In the ideal model, after normalization, this vector takes the
particularly simple form
\[
v_\rho(k)
=
\frac{
\sin\!\left(\frac{L}{2}(\gamma_\rho-\alpha_k)\right)
}{
\frac{L}{2}(\gamma_\rho-\alpha_k)
}.
\]
This model makes the geometry especially transparent. Let us look at a few
examples. \\

\noindent \textbf{A simple on-grid zero.} First suppose that
\[
\rho=\frac12+i\gamma
\]
lies on the critical line and that its ordinate lies exactly on one of the
grid points, say $\gamma=\alpha_j.$ Then
\[
v_\rho=e_j,
\]
where \(e_j\) is the \(j\)-th standard basis vector:
\[
v_\rho=(\ldots,0,0,1,0,0,\ldots).
\]
Thus, in the simplest possible case, a zero on the critical line literally
becomes one coordinate direction. Moreover,
\[
\|v_\rho\|_2=1.
\]

\noindent \textbf{A simple off-grid zero.} Next suppose that the zero is still on the critical line, but lies halfway
between two consecutive grid points,
\[
\gamma=\frac{\alpha_j+\alpha_{j+1}}{2}.
\]
Then the vector is no longer concentrated at one coordinate. Instead,
\[
v_\rho(j+m)
=
\frac{
\sin\!\left(\pi\left(\frac12-m\right)\right)
}{
\pi\left(\frac12-m\right)
},
\]
and hence it looks like
\[
v_\rho
=
(\ldots,
-0.212,\,
0.637,\,
0.637,\,
-0.212,\,
0.127,\ldots).
\]
So the zero is now represented by a sinc packet spread over several
nearby coordinates. It is still a single vector direction; moving the zero
between grid points simply spreads that direction over the neighboring
coordinates. Its total \(L^2\)-mass does not change:
\[
\|v_\rho\|_2=1.
\]

\noindent \textbf{A simple off-line on-grid zero.} Finally, let us move the zero away from the critical line. Write
\[
\rho
=
\frac12+\frac{r}{L}+i\alpha_j,
\]
so that the ordinate is again centered on the grid, while the normalized
horizontal displacement is
\[
r=L\delta.
\]
Take, for example,
\[
r=5.
\]
Writing \(m=k-j\), the \(k\)-th coordinate is complex:
\[
v_\rho(k)= v_\rho(j+m)
=
\frac{\sin(\pi m-ir/2)}
{\pi m-ir/2}.
\]
For \(r=5\), the central part of the vector is approximately
\[
\begin{aligned}
v_\rho
=
(\ldots,
&-0.159-0.600i,\,
 0.331+0.831i,\,
-0.938-1.179i,\\
&2.420,\,
-0.938+1.179i,\,
 0.331-0.831i,\,
-0.159+0.600i,
\ldots).
\end{aligned}
\]

This is the first important difference between an on-line and an off-line
zero. An on-line zero gives a real vector. Moving the zero away from the
critical line turns it into a complex vector. Its reflected zero,
\[
\frac12-\frac{r}{L}+i\alpha_j,
\]
gives the conjugate vector
\[
\overline{v_\rho}.
\]

The norm of the vector also grows once the zero moves away from the critical
line. In this normalized model,
\[
\|v_\rho\|_2^2
=
\frac{\sinh r}{r}.
\]
Thus, for \(r=5\),
\[
\|v_\rho\|_2^2
=
\frac{\sinh 5}{5}
\approx 14.84,
\qquad
\|v_\rho\|_2
\approx 3.85.
\]
\noindent \textbf{What happens to multiple zeros.}
Multiplicity does not change the vector associated with the location of a zero. Rather, the same vector is counted several times. Thus, if a zero has multiplicity \(m\), its contribution to \(G\) is weighted by \(m\).

For example, a simple zero on the critical line located at the grid point \(\alpha_j\) gives

$$
v_\rho=e_j.
$$

If the same zero has multiplicity \(m\), we simply have \(m\) copies of the same vector \(e_j\). Consequently, its contribution to the matrix is

$$
m\,e_je_j^T.
$$

Thus a double zero contributes \(2e_je_j^T\), a triple zero contributes \(3e_je_j^T\), and so on. Multiplicity therefore increases the weight of a direction, but does not create a new direction.\\

\noindent\textbf{Main fact \#1: multiple and off-line zeros waste positive directions.}

A major point in Claude's argument is that a simple zero on the critical line
can contribute one positive direction to \(G\). A multiple zero at the same
location still contributes only one positive direction, even though it is
counted several times. Similarly, a reflected pair of off-line zeros consists
of two zeros, but together they can contribute at most one positive direction.
In this sense, multiple zeros and off-line zeros waste directions.

\medskip

\noindent\textbf{A simple on-line, on-grid zero.}
Let us first consider a zero on the critical line. Its vector \(v_\rho\) is
real, and its contribution to \(G\) is
\[
v_\rho v_\rho^T.
\]
For a real vector \(x\), since \(v_\rho^T x\) and \(x^Tv_\rho\) are real
numbers,
\[
x^T(v_\rho v_\rho^T)x
=
(v_\rho^T x)^2.
\]

If the zero lies exactly on the grid, then \(v_\rho=e_j\), and
\[
x^T(e_je_j^T)x=x_j^2.
\]
So the only independent positive direction contributed by this zero is
\(e_j\).

\medskip

\noindent\textbf{A simple on-line, off-grid zero.}
The same remains true if the zero lies between grid points and \(v_\rho\) is
spread over several coordinates. Suppose, for contradiction, that one zero
could create two independent positive directions. Then there would be a
two-dimensional subspace \(V\) such that
\[
x^T(v_\rho v_\rho^T)x>0
\]
for every non-zero \(x\in V\). But the map
\[
x\longmapsto v_\rho^T x
\]
is a linear map from the two-dimensional space \(V\) to \(\mathbb R\). By
rank-nullity, its kernel has dimension at least one. Therefore there is a
non-zero \(x\in V\) such that
\[
v_\rho^T x=0.
\]
For this vector,
\[
x^T(v_\rho v_\rho^T)x=0,
\]
which is a contradiction. Thus a zero on the critical line can contribute at
most one independent positive direction, whether or not it lies on the grid.

\medskip

\noindent\textbf{A multiple zero on the critical line.}
This also explains why multiplicity wastes directions. If a zero has
multiplicity \(m\), its contribution is
\[
m\,v_\rho v_\rho^T.
\]
The coefficient changes, but the direction does not. Thus a double or triple
zero at the same location still gives only one positive direction.

\medskip

\noindent\textbf{A reflected pair of off-line zeros.}
Now consider a zero away from the critical line,
\[
\rho=\frac12+\frac{r}{L}+i\gamma.
\]
By the functional equation of the Riemann zeta-function, the reflected zero
\[
1-\overline{\rho}
=
\frac12-\frac{r}{L}+i\gamma
\]
is also a zero. If
\[
v_\rho=a+ib,
\]
then the vector corresponding to the reflected zero is
\[
\overline{v_\rho}=a-ib,
\]
where \(a\) and \(b\) are real vectors.

Together the two zeros contribute
\[
v_\rho v_\rho^T
+
\overline{v_\rho}\,\overline{v_\rho}^{\,T}
=
2(aa^T-bb^T).
\]
Thus, for a real vector \(x\),
\[
x^T
\Bigl(
v_\rho v_\rho^T
+
\overline{v_\rho}\,\overline{v_\rho}^{\,T}
\Bigr)
x
=
2(a^Tx)^2-2(b^Tx)^2.
\]

Claude chooses the grid and the test functions so that, by Poisson summation,
\[
v_\rho^Tv_\rho=1.
\]
Importantly, this is true both for zeros on the critical line and for zeros
away from it. 

Here \(v_\rho^Tv_\rho\) is not the Euclidean norm when \(v_\rho\) is complex.
The Euclidean norm is
$v_\rho^*v_\rho,$ which for an off-line zero equals
\[
v_\rho^*v_\rho
=
\|v_\rho\|_2^2
=
\frac{\sinh r}{r},
\]
and may be much larger than \(1\). Thus there is no contradiction between
\[
v_\rho^Tv_\rho=1
\]
and the growth of the Euclidean norm away from the critical line.
Writing again
\[
v_\rho=a+ib,
\]
we obtain
\[
1
=
(a+ib)^T(a+ib)
=
\|a\|^2-\|b\|^2+2i\,a^Tb.
\]
Since the left-hand side is real,
\[
a^Tb=0,
\]
and hence \(a\) and \(b\) are orthogonal. We also obtain
\[
\|a\|^2-\|b\|^2=1.
\]

The reflected pair therefore has one positive eigendirection, given by \(a\),
and one negative eigendirection, given by \(b\). Writing
\[
S_r=\frac{\sinh r}{r},
\]
the corresponding eigenvalues are
\[
\lambda_+=1+S_r,
\qquad
\lambda_-=-(S_r-1).
\]

Thus, despite the complicated appearance of the vector associated with an
off-line zero, the geometry of a reflected pair is very simple: two zeros give
only one positive direction and one negative direction.

\medskip

\noindent\textbf{The counting principle.}
A simple zero on the critical line uses one zero and can give one positive
direction. A multiple zero uses several zeros but still gives only one positive
direction. A reflected off-line pair uses two zeros but also gives only one
positive direction. Multiple zeros and off-line zeros are therefore inefficient
when we count independent positive directions of \(G\).\\

\noindent\textbf{Main fact \#2: Clustering cannot create new positive directions.}

Suppose there are \(p\) reflected off-line pairs. From the discussion above,
the \(j\)-th pair contributes a matrix of the form
\[
\lambda_j a_j a_j^T-\mu_j b_j b_j^T,
\qquad
\lambda_j,\mu_j>0,
\]
where \(a_j\) is its positive direction and \(b_j\) its negative direction.
Let
\[
Q
=
\sum_{j=1}^p
\left(
\lambda_j a_j a_j^T
-
\mu_j b_j b_j^T
\right)
\]
be the total contribution of all off-line pairs.

We claim that \(Q\) can have at most \(p\) independent positive directions.
Suppose, for contradiction, that there were a \((p+1)\)-dimensional subspace
\(V\) such that
\[
x^TQx>0
\]
for every non-zero \(x\in V\).

Consider the linear map
\[
A:V\longrightarrow \mathbb R^p,
\qquad
A(x)
=
\begin{pmatrix}
a_1^Tx\\
\vdots\\
a_p^Tx
\end{pmatrix}.
\]
Since \(\dim V=p+1\), while the target space has dimension \(p\), rank-nullity
implies that
\[
\dim\ker A\geq1.
\]
Hence there exists a non-zero \(x\in V\) such that
\[
a_1^Tx=\cdots=a_p^Tx=0.
\]
For this vector all the positive terms disappear, and therefore
\[
x^TQx
=
-\sum_{j=1}^p
\mu_j(b_j^Tx)^2
\leq0,
\]
contradicting the assumption that \(x^TQx>0\) for every non-zero \(x\in V\). Thus
\[
n_+(Q)\leq p.
\]

This is the key point about clustering. The zeros may cluster and their
positive and negative directions may rotate and interact, but \(p\) reflected
off-line pairs can never create more than \(p\) independent positive
directions.\\

\noindent\textbf{Main fact \#3: The first two moments force \(G\) to have at least \(\frac34 N(T)\) positive directions.}

Let $n_+(G)$ denote the number of positive eigenvalues of \(G\). Equivalently, this is the
largest dimension of a subspace \(V\) such that
\[
x^T G x>0
\]
for every non-zero \(x\in V\).

Suppose the positive eigenvalues of \(G\) are
\[
\lambda_1,\ldots,\lambda_m>0,
\qquad
m=n_+(G),
\]
and write the negative eigenvalues as
\[
-\mu_1,\ldots,-\mu_q,
\qquad
\mu_j>0.
\]

Since
\[
\operatorname{tr}G=(1+o(1))N(T),
\]
we have
\[
\sum_{i=1}^m\lambda_i-\sum_{j=1}^q\mu_j
=
(1+o(1))N(T).
\]
Therefore
\[
\sum_{i=1}^m\lambda_i
\geq
(1+o(1))N(T).
\]

By Cauchy--Schwarz,
\[
\left(\sum_{i=1}^m\lambda_i\right)^2
\leq
m\sum_{i=1}^m\lambda_i^2.
\]
Also,
\[
\sum_{i=1}^m\lambda_i^2
\leq
\operatorname{tr}(G^2)
=
\left(\frac43+o(1)\right)N(T).
\]
Hence
\[
N(T)^2(1+o(1))
\leq
n_+(G)
\left(\frac43+o(1)\right)N(T),
\]
and therefore
\[
\boxed{
n_+(G)
\geq
\left(\frac34-o(1)\right)N(T).
}
\]

\textbf{A thought experiment.} Before introducing the proportion of simple zeros, it is useful to look at an
extreme thought experiment. Suppose that all zeros were off the critical line. Since off-line zeros occur in reflected pairs, there would be approximately ${N(T)}/{2}$ such pairs. By Main fact \#2, these pairs can create at most ${N(T)}/{2}$ positive directions. But the positive eigenvalues must still carry total mass at least
$(1+o(1))N(T).$ If this mass is distributed among at most \(N(T)/2\) positive eigenvalues,
Cauchy--Schwarz gives
\[
\sum_{\lambda_i>0}\lambda_i^2
\geq
\frac{
\left(\sum_{\lambda_i>0}\lambda_i\right)^2
}{
N(T)/2
}
\geq
(2-o(1))N(T).
\]
Therefore
\[
\operatorname{tr}(G^2)
\geq
(2-o(1))N(T),
\]
which contradicts
\[
\operatorname{tr}(G^2)
=
\left(\frac43+o(1)\right)N(T).
\]

\noindent \textbf{Proof of $\%50$ of zeros are on the critical line.} Now let \(s\) denote the number of zeros that are simple and on the critical
line. These \(s\) zeros can contribute at most \(s\) positive directions.

The remaining
\[
N(T)-s
\]
zeros are either multiple zeros on the critical line or zeros away from the
critical line. By Main facts \textbf{\#1} and \textbf{\#2}, every two such zeros can contribute
at most one positive direction. Therefore
\[
n_+(G)
\leq
s+\frac{N(T)-s}{2}
=
\frac{N(T)+s}{2}.
\]

Combining this with
\[
n_+(G)
\geq
\left(\frac34-o(1)\right)N(T),
\]
we obtain
\[
\frac{N(T)+s}{2}
\geq
\left(\frac34-o(1)\right)N(T),
\]
and hence\footnote{A very interesting detail of the Anthropic breakthrough is that one of Claude's sub-agents first arrived at an unconditional \(1/2\) bound using essentially the direction-counting argument described here. The orchestrating agent noted that this would already improve the previous \(5/12\approx41.7\%\) record, but initially regarded the claim with considerable suspicion and sent several other agents to referee it. The following morning, a new sub-agent was explicitly tasked with pushing the \(1/2\) argument further; this led to the stronger \(2/3\) argument used in the final proof.
}
\[
\boxed{
s
\geq
\left(\frac12-o(1)\right)N(T).
}
\]

Thus this elementary direction-counting argument already shows that at least
half of the zeros are simple and lie on the critical line.
\subsection{From \(1/2\) to \(2/3\): keeping track of the eigenvalues}

The \(50\%\) argument only counted positive directions. To reach \(2/3\), we also keep track of the sizes of the corresponding eigenvalues. For simplicity, write
\[
N=N(T).
\]
Let \(s\) be the number of zeros which are simple and lie on the critical line. Decompose
\[
G=P+Q,
\]
where \(P\) is the contribution of these \(s\) simple on-line zeros, and \(Q\) is the contribution of all remaining zeros.

For \(P\), we have
\[
\operatorname{tr}P=s,
\qquad
\operatorname{rank}P\leq s.
\]

For \(Q\),
\[
\operatorname{tr}Q=N-s,
\]
and Main facts \#1 and \#2 give
\[
n_+(Q)\leq \frac{N-s}{2}.
\]

First suppose that \(Q\) has no negative eigenvalues. Then \(Q\) is positive semidefinite and
\[
\operatorname{rank}Q\leq \frac{N-s}{2}.
\]

By Cauchy--Schwarz,
\[
\operatorname{tr}(P^2)
\geq
\frac{(\operatorname{tr}P)^2}{\operatorname{rank}P}
\geq s,
\]
and similarly,
\[
\operatorname{tr}(Q^2)
\geq
\frac{(\operatorname{tr}Q)^2}{\operatorname{rank}Q}
\geq
\frac{(N-s)^2}{(N-s)/2}
=
2(N-s).
\]

Moreover, since \(P\) and \(Q\) are positive semidefinite,
\[
\operatorname{tr}(PQ)\geq0.
\]
We prove this elementary linear-algebra fact in the appendix.

Therefore
\[
\begin{aligned}
\operatorname{tr}(G^2)
&=
\operatorname{tr}(P^2)
+
\operatorname{tr}(Q^2)
+
2\operatorname{tr}(PQ)\\
&\geq
s+2(N-s)\\
&=
2N-s.
\end{aligned}
\]

Since
\[
\operatorname{tr}(G^2)
=
\left(\frac43+o(1)\right)N,
\]
this would immediately imply
\[
s\geq
\left(\frac23-o(1)\right)N.
\]

Thus, if \(Q\) had no negative eigenvalues, we would already be done.

The only difficulty is that \(Q\) may be indefinite. Negative eigenvalues make \(\operatorname{tr}(Q^2)\) larger, but they may also make the cross term
\[
\operatorname{tr}(PQ)
\]
negative. This is the only remaining source of cancellation.

Write
\[
Q=Q_+-Q_-,
\]
where \(Q_+\) and \(Q_-\) are the positive and negative spectral parts of \(Q\). Both are positive semidefinite, and
\[
Q_+Q_-=0.
\]
We prove this elementary spectral fact in the appendix.

Set
\[
B=\operatorname{tr}Q_-.
\]

Since
\[
\operatorname{tr}Q
=
\operatorname{tr}Q_+
-
\operatorname{tr}Q_-
=
N-s,
\]
we have
\[
\operatorname{tr}Q_+=N-s+B.
\]

Thus any negative mass in \(Q\) forces its positive eigenvalues to carry even more total mass.

Now
\[
G=Q_++(P-Q_-).
\]

Using \(Q_+Q_-=0\),
\[
\begin{aligned}
\operatorname{tr}(G^2)
&=
\operatorname{tr}(Q_+^2)
+
\operatorname{tr}\bigl((P-Q_-)^2\bigr)
+
2\operatorname{tr}(PQ_+).
\end{aligned}
\]

Since \(P\) and \(Q_+\) are positive semidefinite,
\[
\operatorname{tr}(PQ_+)\geq0.
\]

Hence
\[
\operatorname{tr}(G^2)
\geq
\operatorname{tr}(Q_+^2)
+
\operatorname{tr}\bigl((P-Q_-)^2\bigr).
\]

The matrix \(Q_+\) has at most
\[
\frac{N-s}{2}
\]
positive eigenvalues, whose sum is \(N-s+B\). Therefore
\[
\operatorname{tr}(Q_+^2)
\geq
\frac{2(N-s+B)^2}{N-s}.
\]

We now divide into two cases.

\medskip

\noindent\textbf{Case 1: \(B>s\).} Then
\[
N-s+B>N.
\]

Consequently,
\[
\operatorname{tr}(Q_+^2)
>
\frac{2N^2}{N-s}
>
2N.
\]

It follows that
\[
\operatorname{tr}(G^2)>2N,
\]
which contradicts
\[
\operatorname{tr}(G^2)
=
\left(\frac43+o(1)\right)N.
\]

Thus the case \(B>s\) is impossible.

\medskip

\noindent\textbf{Case 2: \(B\leq s\).} Set
\[
R=P-Q_-.
\]

We first claim that
\[
n_+(R)\leq s.
\]

Indeed, suppose that \(R\) were positive on an \((s+1)\)-dimensional subspace \(V\). Since
\[
\operatorname{rank}P\leq s,
\]
the linear map \(P:V\to\mathbb R^d\) has a non-zero vector \(x\in V\) in its kernel. For this vector,
\[
x^TRx
=
-x^TQ_-x
\leq0,
\]
contradicting positivity of \(R\) on \(V\).

Thus \(R\) has at most \(s\) positive eigenvalues.

Moreover,
\[
\operatorname{tr}R
=
s-B
\geq0.
\]

Hence the positive eigenvalues of \(R\) have total sum at least \(s-B\). Since there are at most \(s\) of them, Cauchy--Schwarz gives
\[
\operatorname{tr}(R^2)
\geq
\frac{(s-B)^2}{s}.
\]

Combining the two estimates,
\[
\operatorname{tr}(G^2)
\geq
\frac{2(N-s+B)^2}{N-s}
+
\frac{(s-B)^2}{s}.
\]

Expanding the right-hand side,
\begin{equation}
\label{imp_eq}
    \operatorname{tr}(G^2)
\geq
2N-s
+
2B
+
B^2
\left(
\frac{2}{N-s}
+
\frac1s
\right).
\end{equation}

In particular,
\[
\boxed{
\operatorname{tr}(G^2)\geq2N-s.
}
\]

Finally,
\[
2N-s
\leq
\operatorname{tr}(G^2)
=
\left(\frac43+o(1)\right)N,
\]
and therefore
\[
\boxed{
s\geq
\left(\frac23-o(1)\right)N.
}
\]

Thus at least two thirds of the zeros are simple and lie on the critical line.
\subsection{Where did the clustering go?}

It is useful to return to the inequality \eqref{imp_eq}. Recall that
\[
B=\operatorname{tr}Q_-
\]
measures the total mass of the negative eigenvalues of \(Q\). Every term involving
\(B\) in \eqref{imp_eq} is positive. Thus any quantitative lower bound on \(B\)
would translate directly into an improvement of Claude's \(2/3\) result.

This raises a natural question. Every reflected pair of zeros off the critical
line contributes one positive and one negative direction. Why can we not simply
add these negative directions and obtain a large value of \(B\)?

The negative eigenvalues of the individual reflected pairs are not additive.
If the \(j\)-th pair contributes
\[
\lambda_j a_j a_j^T-\mu_j b_j b_j^T,
\]
then
\[
Q
=
\sum_j
\left(
\lambda_j a_j a_j^T-\mu_j b_j b_j^T
\right)
\]
also contains the positive directions \(a_j\). The positive direction coming
from one pair may point partly, or even almost entirely, in a negative direction
coming from another pair. After the matrices are added, these contributions can
cancel.

This is where the clustering obstruction survives. In the most effective
clustering scenario, the negative spectral mass can be almost completely
cancelled, so that \(B\) is close to zero. The bad zeros then leave essentially
one positive direction for every two zeros, carrying eigenvalue about \(2\). In
this extremal sense, \(Q\) behaves like \(2I\) on its surviving positive
subspace.

Thus Claude's argument does not eliminate clustering. It proves the \(2/3\)
result even in the worst case where clustering destroys essentially all of the
negative spectral mass. To go beyond \(2/3\), one would need to show that this
cancellation cannot be complete, so that some positive lower bound for \(B\)
must survive.
\section{The Recursive Identity \& Proof.} For \(k\geq1\), define
\[
\mathcal C_k(t)
:=
-\frac{2}{L}\sum_{n<T}\frac{\Lambda(n)}{\sqrt n}
\left(1-\frac{\log n}{L}\right)^k\cos(t\log n)
\]
and
\[
\Im_k(t)
:=
\frac{2}{L}\sum_{n<T}\frac{\Lambda(n)}{\sqrt n}
\left(1-\frac{\log n}{L}\right)^k\sin(t\log n).
\]

\begin{lemma}[Recursive identity]
\label{lem:recursion}
For every \(k\geq1\),
\[
\Im_k(t)-\Im_{k+1}(t)
=
\frac{1}{L}\mathcal C_k'(t).
\]
Consequently, for every integer \(M\geq1\),
\[
\Im Z\!\left(\frac12+it\right)
=
\frac1L\sum_{k=1}^M\mathcal C_k'(t)
+
\Im_{M+1}(t).
\]
Moreover, for \(M=\lceil L\rceil\), then
\[
\Im_{M+1}(t)\ll \frac1L .
\]
\end{lemma}

\subsection{Higher-order Fej\'er kernels.} Put
\[
\widetilde S_k(z)
:=
2\int_0^1(1-u)^k\cos(2uz)\,du.
\]

\begin{lemma}[Higher-order Fej\'er identity]
\label{lem:Ck-zero}
Let \(T>e\), \(L=\log T\), and \(t\asymp T\). Then, for every integer
\(k\geq1\),
\[
\mathcal C_k(t)
=
-1+
\Re\sum_{\rho}
\widetilde S_k\!\left(
-\frac{iL}{2}
\left(\rho-\frac12-it\right)
\right)
+
O_k\!\left(\frac1L\right).
\]
Writing \(\rho=\beta+i\gamma\), this is
\[
\mathcal C_k(t)
=
-1+
\Re\sum_{\rho}
\widetilde S_k\!\left(
\frac{L}{2}
\left(
(\gamma-t)-i\left(\beta-\frac12\right)
\right)
\right)
+
O_k\!\left(\frac1L\right).
\]
More precisely, for \(1\leq k\leq 2L\), the error term may be
written as \(E_k(t)\), where
\[
E_k(t)\ll \frac1L,
\qquad
E_k'(t)\ll \frac1{TL},
\]
uniformly for \(t\asymp T\).
\end{lemma}

\begin{lemma}[Odd-kernel tail identity]
\label{lem:odd-tail}
For an integer \(M\geq1\), put
\[
B_M(z):=A_\infty(z)-A_M(z)
=2\int_0^1(1-u)^{M+1}\sin(uz)\,du.
\]
Then, uniformly for \(t\asymp T\),
\[
\Re\sum_\rho
B_M\!\left(-iL\left(\rho-\frac12-it\right)\right)
=
\Im_{M+1}(t)+O\!\left(\frac1L\right),
\]
where the zero sum is taken symmetrically.  In particular, if
\(M\geq(1/2+\varepsilon)L\), then the left-hand side is
\(O_\varepsilon(L^{-1})\).
\end{lemma}

\subsection{Proof of the main theorem}

\begin{proof}[Proof of Theorem~\ref{thm:ImZ-limit-kernel}]
Put
\[
z_\rho(t)
:=
-iL\left(\rho-\frac12-it\right).
\]
By Lemma~\ref{lem:recursion},
\[
\Im Z\!\left(\frac12+it\right)
=
\frac1L\sum_{k=1}^{M}\mathcal C_k'(t)
+
\Im_{M+1}(t).
\]
Choose \(M\asymp L\).  Then
\[
\Im_{M+1}(t)\ll \frac1L.
\]

By Lemma~\ref{lem:Ck-zero},
\[
\mathcal C_k(t)
=
-1+
\Re\sum_{\rho}
\widetilde S_k\!\left(\frac{z_\rho(t)}2\right)
+
E_k(t),
\]
where
\[
E_k'(t)\ll \frac1{TL}
\]
uniformly for \(1\leq k\leq M\).
Since
\[
z_\rho'(t)=-L,
\]
we obtain
\[
\frac1L\mathcal C_k'(t)
=
-\frac12
\Re\sum_{\rho}
\widetilde S_k'\!\left(\frac{z_\rho(t)}2\right)
+
\frac1L E_k'(t).
\]
Hence
\[
\frac1L\sum_{k=1}^{M}\mathcal C_k'(t)
=
\Re\sum_{\rho}A_M\bigl(z_\rho(t)\bigr)
+
O\!\left(\frac1{TL}\right),
\]
where
\[
A_M(z)
:=
-\frac12\sum_{k=1}^{M}
\widetilde S_k'\!\left(\frac z2\right).
\]

Using
\[
\widetilde S_k\!\left(\frac z2\right)
=
2\int_0^1(1-u)^k\cos(uz)\,du,
\]
we find
\[
\begin{aligned}
A_M(z)
&=
2\int_0^1
u\sum_{k=1}^{M}(1-u)^k\sin(uz)\,du\\
&=
2\int_0^1
(1-u)\bigl(1-(1-u)^M\bigr)\sin(uz)\,du.
\end{aligned}
\]
It remains to replace \(A_M\) by \(A_\infty\).  By
Lemma~\ref{lem:odd-tail},
\[
\Re\sum_{\rho}
\bigl(A_\infty-A_M\bigr)
\!\left(-iL\left(\rho-\frac12-it\right)\right)
=O\!\left(\frac1L\right)
\]
for our choice \(M\asymp L\).  Hence
\[
\Re\sum_{\rho}
A_M\!\left(-iL\left(\rho-\frac12-it\right)\right)
=
\Re\sum_{\rho}
A_\infty\!\left(-iL\left(\rho-\frac12-it\right)\right)
+O\!\left(\frac1L\right).
\]

Consequently,
\[
\Im Z\!\left(\frac12+it\right)
=
\Re\sum_{\rho}
A_\infty\!\left(
-iL\left(\rho-\frac12-it\right)
\right)
+
O\!\left(\frac1L\right).
\]
This is unconditional.  If the Riemann Hypothesis is assumed, then
\(\rho=\frac12+i\gamma\), and the identity reduces to
\[
\Im Z\!\left(\frac12+it\right)
=
\sum_{\gamma}
A_\infty\!\big((\gamma-t)L\big)
+
O\!\left(\frac1L\right).
\]
\end{proof}

\subsection{Proofs of the auxiliary lemmas}

\begin{proof}[Proof of Lemma~\ref{lem:recursion}]
Differentiating \(\mathcal C_k(t)\), we obtain
\[
\frac1L\mathcal C_k'(t)
=
\frac{2}{L}\sum_{n<T}\frac{\Lambda(n)}{\sqrt n}
\frac{\log n}{L}
\left(1-\frac{\log n}{L}\right)^k
\sin(t\log n).
\]
On the other hand,
\[
\begin{aligned}
\Im_k(t)-\Im_{k+1}(t)
&=
\frac{2}{L}\sum_{n<T}\frac{\Lambda(n)}{\sqrt n}
\left[
\left(1-\frac{\log n}{L}\right)^k
-
\left(1-\frac{\log n}{L}\right)^{k+1}
\right]
\sin(t\log n)\\
&=
\frac{2}{L}\sum_{n<T}\frac{\Lambda(n)}{\sqrt n}
\frac{\log n}{L}
\left(1-\frac{\log n}{L}\right)^k
\sin(t\log n),
\end{aligned}
\]
which proves
\[
\Im_k(t)-\Im_{k+1}(t)=\frac1L\mathcal C_k'(t).
\]
Summing over \(1\leq k\leq M\) gives the telescoping identity.

Finally,
\[
\left(1-\frac{\log n}{L}\right)^{M+1}
\leq
n^{-(M+1)/L},
\]
and hence, if \(M\geq(1/2+\varepsilon)L\),
\[
|\Im_{M+1}(t)|
\leq
\frac2L
\sum_{n\geq2}
\frac{\Lambda(n)}{n^{1+\varepsilon}}
\ll_\varepsilon\frac1L.
\]
\end{proof}

\begin{proof}[Proof of Lemma~\ref{lem:odd-tail}]
Set \(s_0=\frac12+it\) and introduce the odd entire transform
\[
H_M(w):=L B_M(-iLw)
=-2iL\int_0^1(1-u)^{M+1}\sinh(Luw)\,du.
\]
Apply to \(H_M(s-s_0)\) the same two-line contour argument used in
Lemma~\ref{lem:Ck-zero}.  On the line \(\Re s>1\), expanding
\(\zeta'/\zeta\) into its Dirichlet series and writing the hyperbolic
sine as a difference of exponentials gives the prime polynomial with
weight \((1-\log n/L)^{M+1}\).  Reflecting the left-hand line by the
functional equation gives its complex conjugate.  Taking the real part
of the residue identity therefore yields
\[
\Re\sum_\rho B_M\!\left(-iL(\rho-s_0)\right)
=
\frac{2}{L}\sum_{n<T}\frac{\Lambda(n)}{\sqrt n}
\left(1-\frac{\log n}{L}\right)^{M+1}
\sin(t\log n)
+O\!\left(\frac1L\right).
\]
The sum on the right is precisely \(\Im_{M+1}(t)\).  The pole and
Gamma-factor terms are \(O(L^{-1})\), by the same Stirling estimates as
in Lemma~\ref{lem:Ck-zero}.  Finally, if
\(M\geq(1/2+\varepsilon)L\), Lemma~\ref{lem:recursion} gives
\(\Im_{M+1}(t)\ll_\varepsilon L^{-1}\).
\end{proof}

\begin{proof}[Proof of Lemma~\ref{lem:Ck-zero}]
This is a higher-order extension of the unconditional Fej\'er-kernel
identity proved by the author in \cite{Ary22}, corresponding to the case
\(k=1\).  Under the Riemann Hypothesis, the \(k=1\) identity was proved
earlier by Gonek \cite{Gonek}. \\

We briefly describe the architecture of the argument.  We begin with a
contour on the right of the critical strip, where the logarithmic derivative
of the zeta function can be expanded as an absolutely convergent Dirichlet
series.  We then move the contour to the left of the critical strip.  The
difference between the two contour integrals is given by the residues of the
non-trivial zeros, together with the contribution of the pole of the zeta
function.  On the left-hand contour we use the functional equation for the
logarithmic derivative to reflect the zeta term back to the right-hand
half-plane, where it can once again be expanded into a Dirichlet series.
The remaining term arising from the functional equation is the Gamma-factor
contribution.  This is treated separately by Stirling's formula, exactly as
in the \(k=1\) argument.  In this way no assumption on the location of the
non-trivial zeros is required.\\

Put
\[
s_0:=\frac12+it
\]
and, for \(k\geq1\), define
\begin{equation}
\label{eq:Gk-def}
G_k(w)
:=
L\int_{-1}^{1}(1-|u|)^k e^{Luw}\,du.
\end{equation}
The function \(G_k\) is entire and even.  Moreover,
\begin{equation}
\label{eq:Gk-Sk-complex}
\frac1L G_k(w)
=
\widetilde S_k\!\left(-\frac{iLw}{2}\right).
\end{equation}
Indeed,
\[
\frac1L G_k(w)
=
2\int_0^1(1-u)^k\cosh(Luw)\,du
=
2\int_0^1(1-u)^k
\cos\!\left(-iLuw\right)\,du.
\]

Let \(c=1+\delta\), where \(0<\delta<1\), and consider
\[
I_k^{+}(t)
:=
\frac{1}{2\pi i}
\int_{(c)}
\frac{\zeta'}{\zeta}(s)
G_k(s-s_0)\,ds.
\]
Since \(c>1\), we may use the absolutely convergent expansion
\[
\frac{\zeta'}{\zeta}(s)
=
-\sum_{n=2}^{\infty}\frac{\Lambda(n)}{n^s}.
\]
For \(n>1\), Perron's formula gives
\begin{equation}
\label{eq:Gk-Perron}
\frac{1}{2\pi i}
\int_{(c)}
n^{-s}G_k(s-s_0)\,ds
=
\begin{cases}
\displaystyle
n^{-s_0}
\left(1-\frac{\log n}{L}\right)^k,
& n<T,\\[3mm]
0,&n>T.
\end{cases}
\end{equation}
Consequently,
\begin{equation}
\label{eq:Iplus}
I_k^{+}(t)
=
-\sum_{n<T}\frac{\Lambda(n)}{n^{s_0}}
\left(1-\frac{\log n}{L}\right)^k.
\end{equation}
If we write
\[
P_k(t)
:=
\sum_{n<T}\frac{\Lambda(n)}{n^{s_0}}
\left(1-\frac{\log n}{L}\right)^k,
\]
then
\begin{equation}
\label{eq:Ck-P}
\mathcal C_k(t)
=
-\frac{2}{L}\Re P_k(t),
\qquad
I_k^{+}(t)=-P_k(t).
\end{equation}

We next introduce the corresponding integral on the line
\(\Re s=1-c=-\delta\):
\[
I_k^{-}(t)
:=
\frac{1}{2\pi i}
\int_{(1-c)}
\frac{\zeta'}{\zeta}(s)
G_k(s-s_0)\,ds.
\]
By the residue theorem, the difference between the two contour
integrals is
\begin{equation}
\label{eq:residue-strip}
I_k^{+}(t)-I_k^{-}(t)
=
\sum_{\rho}G_k(\rho-s_0)
-
G_k(1-s_0).
\end{equation}
Here the sum is over all non-trivial zeros of \(\zeta\), counted with
multiplicity.  The second term comes from the pole of \(\zeta\) at
\(s=1\), whose residue in \(\zeta'/\zeta\) is \(-1\).
Since \(0<\delta<1\), no trivial zero lies between the two vertical
lines. \\

We now treat \(I_k^{-}(t)\) by the functional equation.  Using
\begin{equation}
\label{eq:FE-logder}
\frac{\zeta'}{\zeta}(s)
=
\frac{\chi'}{\chi}(s)
-
\frac{\zeta'}{\zeta}(1-s),
\end{equation}
we obtain
\begin{equation}
\label{eq:Iminus-FE}
I_k^{-}(t)
=
\Gamma_k(t)-J_k(t),
\end{equation}
where
\[
\Gamma_k(t)
:=
\frac{1}{2\pi i}
\int_{(1-c)}
\frac{\chi'}{\chi}(s)
G_k(s-s_0)\,ds
\]
and
\[
J_k(t)
:=
\frac{1}{2\pi i}
\int_{(1-c)}
\frac{\zeta'}{\zeta}(1-s)
G_k(s-s_0)\,ds.
\]

In \(J_k(t)\), make the change of variables \(z=1-s\).  Since \(G_k\)
is even and
\[
1-s_0=\frac12-it=\overline{s_0},
\]
we obtain
\[
J_k(t)
=
\frac{1}{2\pi i}
\int_{(c)}
\frac{\zeta'}{\zeta}(z)
G_k(z-\overline{s_0})\,dz.
\]
We may again expand \(\zeta'/\zeta\) into its Dirichlet series, and the
same calculation as in \eqref{eq:Gk-Perron} gives
\[
J_k(t)
=
-\sum_{n<T}\frac{\Lambda(n)}{n^{\overline{s_0}}}
\left(1-\frac{\log n}{L}\right)^k
=
-\overline{P_k(t)}.
\]
Thus
\begin{equation}
\label{eq:Iminus-final}
I_k^{-}(t)
=
\Gamma_k(t)+\overline{P_k(t)}.
\end{equation}

Substituting \eqref{eq:Iplus} and \eqref{eq:Iminus-final} into
\eqref{eq:residue-strip}, we find
\[
-P_k(t)-\overline{P_k(t)}-\Gamma_k(t)
=
\sum_{\rho}G_k(\rho-s_0)-G_k(1-s_0).
\]
Taking real parts,
\[
-2\Re P_k(t)
=
\Re\sum_{\rho}G_k(\rho-s_0)
-
\Re G_k(1-s_0)
+
\Re\Gamma_k(t).
\]
Therefore, by \eqref{eq:Ck-P},
\begin{equation}
\label{eq:Ck-pre-gamma}
\mathcal C_k(t)
=
\frac1L
\Re\sum_{\rho}G_k(\rho-s_0)
-
\frac1L\Re G_k(1-s_0)
+
\frac1L\Re\Gamma_k(t).
\end{equation}

It remains to evaluate the Gamma-factor contribution.  As in the proof
of the corresponding lemmas in the Landau--Gonek formula, we move the
line in \(\Gamma_k(t)\) from \(\Re s=1-c\) to \(\Re s=1/2\).  The
possible contribution from the singularity at \(s=0\) is negligible for
\(t\asymp T\).  Hence
\[
\Gamma_k(t)
=
\frac{1}{2\pi}
\int_{-\infty}^{\infty}
\frac{\chi'}{\chi}\left(\frac12+i\tau\right)
G_k\bigl(i(\tau-t)\bigr)\,d\tau
+
O_k(T^{-1}).
\]
Using
\[
G_k(iv)
=
L\widetilde S_k\left(\frac{Lv}{2}\right)
\]
and making the change of variables
\[
x=\frac{L(\tau-t)}2,
\]
we obtain
\begin{equation}
\label{eq:Gamma-Sk}
\Gamma_k(t)
=
\frac1\pi
\int_{-\infty}^{\infty}
\widetilde S_k(x)
\frac{\chi'}{\chi}
\left(
\frac12+i\left(t+\frac{2x}{L}\right)
\right)\,dx
+
O_k(T^{-1}).
\end{equation}

On the critical line, For \(|y|\geq2\), Stirling's formula gives
\[
\frac{\chi'}{\chi}\left(\frac12+iy\right)
=
-\log\frac{|y|}{2\pi}
+
O\left(\frac{1}{1+y^2}\right).
\]
Moreover,
\[
\int_{-\infty}^{\infty}\widetilde S_k(x)\,dx=\pi,
\]
while
\[
\widetilde S_k(x)\ll_k \frac{1}{1+x^2}.
\]
It follows from \eqref{eq:Gamma-Sk}, uniformly for \(t\asymp T\), that
\begin{equation}
\label{eq:Gamma-main}
\Gamma_k(t)
=
-\log\frac{t}{2\pi}+O_k(1).
\end{equation}
Moreover, these estimates may be made uniform for
\(1\leq k\leq 2L\).  For real \(x\),
\[
0\leq \widetilde S_k(x)
\ll
\min\left\{\frac{1}{k+1},\frac{k}{1+x^2}\right\},
\qquad
\int_{\mathbb R}\widetilde S_k(x)\,dx=\pi.
\]
Repeating the preceding argument after differentiation with respect to
\(t\), and using Stirling's formula once more, gives
\[
\Gamma_k'(t)\ll \frac1T
\]
uniformly for \(1\leq k\leq 2L\) and \(t\asymp T\).  Similarly, from
\eqref{eq:Gk-def}, by integration by parts,
\[
G_k(1-s_0)=O(1),
\qquad
\frac{d}{dt}G_k(1-s_0)\ll\frac1T,
\]
uniformly in the same range.

Now define
\[
E_k(t)
:=
1-\frac1L\Re G_k(1-s_0)
+\frac1L\Re\Gamma_k(t).
\]
Since
\[
\Gamma_k(t)
=
-\log\frac{t}{2\pi}+O(1)
\]
uniformly in this range, and \(t\asymp T\), we have
\[
E_k(t)\ll\frac1L,
\qquad
E_k'(t)\ll\frac1{TL}.
\]
Thus \eqref{eq:Ck-pre-gamma} becomes
\[
\mathcal C_k(t)
=
-1+
\frac1L\Re\sum_{\rho}G_k(\rho-s_0)
+
E_k(t).
\]
Finally, by \eqref{eq:Gk-Sk-complex},
\[
\frac1L G_k(\rho-s_0)
=
\widetilde S_k\left(
-\frac{iL}{2}(\rho-s_0)
\right),
\]
and therefore
\[
\mathcal C_k(t)
=
-1+
\Re\sum_{\rho}
\widetilde S_k\left(
-\frac{iL}{2}
\left(\rho-\frac12-it\right)
\right)
+
E_k(t),
\]
where
\[
E_k(t)\ll\frac1L,
\qquad
E_k'(t)\ll\frac1{TL}.
\]
This proves the lemma.
\end{proof}

\section{The Imaginary-Part Identity at a Shorter Scale}

We record the form of the imaginary-part identity needed below.
Fix \(0<\alpha\leq1\), put
\[
L:=\log T,
\qquad
L_\alpha:=\alpha L=\log T^\alpha,
\]
and define
\[
Z_\alpha(s)
:=
-\frac{2}{L_\alpha}
\sum_{n<T^\alpha}
\frac{\Lambda(n)}{n^s}
\left(
1-\frac{\log n}{L_\alpha}
\right).
\]

\begin{proposition}[Short-truncation imaginary identity]
\label{prop:short-imaginary}
Let \(0<\alpha\leq1\) be fixed and \(t\asymp T\). Then
\[
\Im Z_\alpha\!\left(\frac12+it\right)
=
\Re\sum_\rho
A_\infty\!\left(
-iL_\alpha
\left(\rho-\frac12-it\right)
\right)
+
O_\alpha\!\left(\frac1L\right),
\]
where
\[
A_\infty(z)
=
2\int_0^1(1-u)\sin(uz)\,du
=
\frac{2(z-\sin z)}{z^2}.
\]
The sum over the non-trivial zeros is understood symmetrically.

In particular, under the Riemann Hypothesis,
\[
\Im Z_\alpha\!\left(\frac12+it\right)
=
\sum_\gamma
A_\infty\!\left(
\alpha L(\gamma-t)
\right)
+
O_\alpha\!\left(\frac1L\right).
\]
\end{proposition}

\begin{proof}
The proof is the same as that of the full-length identity, with replaced throughout by $\alpha L.$
\end{proof}


\begin{proof}[Proof Corollary \ref{cor:shifted-imaginary} ]
Put \(L=\log T\). From the definition of \(\Im_\beta\),
\[
\Im_\beta(t)
=
\frac{2}{\beta L}
\sum_{n<T^\beta}
\frac{\Lambda(n)}{\sqrt n}
\left(
1-\frac{\log n}{\beta L}
\right)
\sin(t\log n).
\]
By the Landau--Gonek formula,
\[
\frac1{N(T)}
\sum_{0<\gamma<T}n^{i\gamma}
\sim
-\frac{\Lambda(n)}{\sqrt n\,L}
\]
in the range \(n<T^\beta\). Hence the shift by
\(-2\pi d/L\) gives
\[
\frac{1}{N(T)}
\sum_{0<\gamma<T}
\Im_\beta\!\left(
\gamma-\frac{2\pi d}{L}
\right)
\sim
\frac{2}{\beta L^2}
\sum_{n<T^\beta}
\frac{\Lambda(n)^2}{n}
\left(
1-\frac{\log n}{\beta L}
\right)
\sin\!\left(
2\pi d\frac{\log n}{L}
\right).
\]
Using
\[
\sum_{n\le x}\Lambda(n)^2\sim x\log x
\]
and partial summation, with
\[
u=\frac{\log n}{L},
\]
the last expression tends to
\[
\frac{2}{\beta}
\int_0^\beta
u\left(1-\frac{u}{\beta}\right)
\sin(2\pi du)\,du.
\]
The closed form follows by an elementary integration.
\end{proof}
\bibliographystyle{plain}

\appendix
\section{An Example of the Clustering Obstruction}
\label{app:clustering-example}

Consider the Fej\'er kernel
\[
K(z)
=
\left(\frac{\sin z}{z}\right)^2
=
\int_{-1}^{1}(1-|u|)e^{2iuz}\,du.
\]
Suppose that the off-line zeros occur in reflected pairs
\[
\rho_j^\pm
=
\frac12\pm\delta_j+i\gamma_j.
\]
We first consider the contribution obtained by summing over all ordered pairs of off-line zeros, including the diagonal terms. Substituting
\[
z=\frac{\rho-\rho'}{2i}L,
\qquad L=\log T,
\]
into the integral representation of \(K\), we obtain
\[
e^{2iuz}=e^{uL(\rho-\rho')}.
\]
Hence
\[
\sum_{\substack{\rho,\rho'\\ \text{off line}}}
K\!\left(
\frac{\rho-\rho'}{2i}L
\right)
=
\int_{-1}^{1}(1-|u|)
\left(\sum_\rho e^{uL\rho}\right)
\left(\sum_{\rho'}e^{-uL\rho'}\right)\,du.
\]

We now group each zero with its reflection. For the \(j\)-th reflected pair,
\[
e^{uL\rho_j^+}+e^{uL\rho_j^-}
=
e^{uL/2}\,
2\cosh(uL\delta_j)e^{iuL\gamma_j}.
\]
Therefore
\[
\sum_\rho e^{uL\rho}
=
e^{uL/2}
\sum_j
2\cosh(uL\delta_j)e^{iuL\gamma_j}.
\]
Similarly,
\[
\sum_{\rho'}e^{-uL\rho'}
=
e^{-uL/2}
\sum_j
2\cosh(uL\delta_j)e^{-iuL\gamma_j}.
\]
The factors \(e^{uL/2}\) and \(e^{-uL/2}\) cancel, and the two remaining sums are complex conjugates. Thus
\[
\boxed{
\sum_{\substack{\rho,\rho'\\ \text{off line}}}
K\!\left(
\frac{\rho-\rho'}{2i}L
\right)
=
\int_{-1}^{1}(1-|u|)
\left|
\sum_j
2\cosh(uL\delta_j)e^{iuL\gamma_j}
\right|^2du.
}
\]

This explains the positivity of the full off-line/off-line contribution when the diagonal terms are included.

For the unconditional Montgomery identity, however, it is useful to separate the diagonal terms. If there are \(p\) simple reflected pairs, then there are \(2p\) off-line zeros, and every diagonal term has
\[
K(0)=1.
\]
Consequently,
\[
\boxed{
\sum_{\substack{\rho\neq\rho'\\ \text{off line}}}
K\!\left(
\frac{\rho-\rho'}{2i}L
\right)
=
\int_{-1}^{1}(1-|u|)
\left|
\sum_j
2\cosh(uL\delta_j)e^{iuL\gamma_j}
\right|^2du
-
2p.
}
\]
The absolute-square representation alone only gives non-negativity before the diagonal terms are removed. The matrix formulation gives a stronger conclusion.\\

Let \(Q_{\mathrm{off}}\) denote the total matrix contribution of these \(p\) reflected off-line pairs. Then
\[
\operatorname{tr}(Q_{\mathrm{off}})=2p,
\]
and, as explained in Section~\ref{Expo},
\[
n_+(Q_{\mathrm{off}})\leq p.
\]
If its positive eigenvalues are
\[
\lambda_1,\ldots,\lambda_m,
\qquad m\leq p,
\]
then
\[
\sum_{i=1}^m\lambda_i
\geq
\operatorname{tr}(Q_{\mathrm{off}})
=
2p,
\]
since the presence of negative eigenvalues can only increase the total mass carried by the positive eigenvalues. Hence, by Cauchy--Schwarz,
\[
\operatorname{tr}(Q_{\mathrm{off}}^2)
\geq
\sum_{i=1}^m\lambda_i^2
\geq
\frac{\left(\sum_{i=1}^m\lambda_i\right)^2}{m}
\geq
4p.
\]
The quantity \(\operatorname{tr}(Q_{\mathrm{off}}^2)\) is precisely the full off-line/off-line sum considered above. Hence
\[
\boxed{
\int_{-1}^{1}(1-|u|)
\left|
\sum_j
2\cosh(uL\delta_j)e^{iuL\gamma_j}
\right|^2du
\geq
4p.
}
\]

Since the \(2p\) diagonal terms each contribute \(K(0)=1\), we obtain
\[
\boxed{
\sum_{\substack{\rho\neq\rho'\\
\rho,\rho'\ \mathrm{off\ line}}}
K\!\left(
\frac{\rho-\rho'}{2i}L
\right)
\geq
2p.
}
\]

Thus the interactions among the off-line zeros themselves are not merely non-negative in aggregate. Even after the diagonal terms are removed, they leave a positive contribution of at least \(2p\).\footnote{This lower bound follows from Claude's matrix argument. It was not available to the author in the earlier scalar analysis of the off-line/off-line contribution.}

The difficulty appears when the off-line zeros interact with zeros on the critical line. Consider one reflected pair
\[
\rho_\pm
=
\frac12\pm\delta+i\gamma
\]
and one zero
\[
\rho_0
=
\frac12+i\gamma_0
\]
on the critical line. Put
\[
x=\frac{\gamma-\gamma_0}{2}L,
\qquad
y=\frac{\delta}{2}L.
\]
There are four ordered mixed interactions,
\[
(\rho_+,\rho_0),\qquad
(\rho_0,\rho_+),\qquad
(\rho_-,\rho_0),\qquad
(\rho_0,\rho_-),
\]
whose total contribution is
\[
4\operatorname{Re}K(x+iy).
\]

Unlike the off-line/off-line contribution, this quantity need not be positive. For example, if the vertical separation is one average zero spacing,
\[
\gamma-\gamma_0
=
\frac{2\pi}{\log T},
\]
then \(x=\pi\), and
\[
4\operatorname{Re}K(\pi+iy)
=
-4\sinh^2(y)
\frac{\pi^2-y^2}{(\pi^2+y^2)^2}.
\]
Consequently, the mixed contribution is negative whenever
\[
0<y<\pi.
\]

Thus an off-line reflected pair can interact negatively with a nearby zero on the critical line. Similar negative interactions occur at other microscopic separations as well.

This is the essential difficulty. The off-line zeros themselves give a definite positive contribution, but their interactions with zeros on the critical line are not sign-definite. Moreover, if a critical-line zero has multiplicity \(m\), the corresponding mixed interaction is multiplied by \(m\). Thus clustering and multiplicity can amplify precisely those interactions which have the wrong sign.

In the unconditional scalar argument, we were not able to rule out configurations in which the accumulation of such negative mixed terms cancels enough of the positive contribution coming from the off-line zeros to interfere with Montgomery's multiplicity argument. Consequently, unlike under the Riemann Hypothesis, one cannot simply discard the off-diagonal terms and obtain an upper bound for the second moment of the multiplicities.

The matrix picture gives a more precise description of when this potential obstruction can occur. Let \(P_{\mathrm{on}}\) denote the total matrix contribution of the zeros on the critical line, and write
\[
Q_{\mathrm{off}}
=
Q_{\mathrm{off},+}
-
Q_{\mathrm{off},-},
\]
where \(Q_{\mathrm{off},+}\) and \(Q_{\mathrm{off},-}\) are the positive and negative spectral parts of the off-line matrix. Since both \(P_{\mathrm{on}}\) and \(Q_{\mathrm{off},+}\) are positive semidefinite,
\[
\operatorname{tr}
\bigl(
P_{\mathrm{on}}Q_{\mathrm{off},+}
\bigr)
\geq0.
\]
Therefore
\[
\operatorname{tr}
\bigl(
P_{\mathrm{on}}Q_{\mathrm{off}}
\bigr)
=
\operatorname{tr}
\bigl(
P_{\mathrm{on}}Q_{\mathrm{off},+}
\bigr)
-
\operatorname{tr}
\bigl(
P_{\mathrm{on}}Q_{\mathrm{off},-}
\bigr)
\]
can be negative only if
\[
Q_{\mathrm{off},-}\neq0.
\]

In particular, if the off-line pairs cluster in such a way that their negative spectrum is completely cancelled, so that
\[
Q_{\mathrm{off},-}=0,
\]
then \(Q_{\mathrm{off}}\) is positive semidefinite and the total interaction between the on-line and off-line contributions is automatically non-negative.

Thus the dangerous situation is more restricted than the existence of individual negative kernel values might suggest. For the mixed terms to create a genuine obstruction, some negative spectral mass of the collective off-line contribution must survive, and the on-line zeros must interact sufficiently strongly with these negative directions.

The issue becomes especially relevant when one tries to control multiple zeros. The diagonal term in the unconditional identity is
\[
\sum_\rho m_\rho^2,
\]
which may be substantially larger than \(N(T)\), while the negative mixed interactions are weighted by the corresponding multiplicities. To recover Montgomery's argument directly from the scalar identity, one would therefore need a sufficiently strong lower bound showing that these mixed interactions cannot compensate for the excess diagonal mass produced by multiple zeros. Such a bound was not available in the unconditional argument.

Claude's matrix method does not show that the mixed scalar terms are positive. Instead, it keeps track of the negative spectral directions responsible for such terms. If no negative spectrum survives, then the mixed contribution is automatically non-negative; if negative spectrum does survive, its spectral mass can itself be used in the trace argument. This information is lost when all pairwise interactions are collapsed into a single scalar sum.

\section{Elementary linear-algebra facts}

We record two standard facts used in section \ref{Expo}. First, if \(P\) and \(Q\) are positive semidefinite, then
\[
\operatorname{tr}(PQ)\ge 0.
\]
Indeed, diagonalizing \(P\), we may write
\[
P=\sum_i \lambda_i u_i u_i^T,
\qquad \lambda_i\ge0.
\]
Therefore
\[
\operatorname{tr}(PQ)
=
\sum_i \lambda_i\,u_i^TQu_i
\ge0,
\]
since \(Q\) is positive semidefinite.\\

Second, every real symmetric matrix \(Q\) admits a decomposition
\[
Q=Q_+-Q_-,
\]
where \(Q_+\) and \(Q_-\) are positive semidefinite and
\[
Q_+Q_-=0.
\]
Indeed, by the spectral theorem,
\[
Q=\sum_i \lambda_i u_i u_i^T
\]
for an orthonormal basis of eigenvectors \(u_i\). Define
\[
Q_+
=
\sum_{\lambda_i>0}\lambda_i u_i u_i^T,
\qquad
Q_-
=
\sum_{\lambda_i<0}(-\lambda_i)u_i u_i^T.
\]
Then both \(Q_+\) and \(Q_-\) are positive semidefinite and
\(Q=Q_+-Q_-\). Moreover, they act on orthogonal eigenspaces, so
\[
Q_+Q_-=0.
\]
\end{document}